\documentclass[12pt]{article}
\usepackage{amssymb,amsthm,amsmath,latexsym}
\newtheorem{thm}{Theorem}
\newtheorem{prop}[thm]{Proposition}
\newtheorem{lem}[thm]{Lemma}
\newtheorem{cor}[thm]{Corollary}
\theoremstyle{remark}
\newtheorem{rem}[thm]{Remark}
\newcommand{\FF}{\mathbb{F}}

\newcommand{\LL}{\mathcal{L}}

\DeclareMathOperator{\Aut}{Aut}
\DeclareMathOperator{\wt}{wt}
\DeclareMathOperator{\supp}{supp}

\begin{document}
\title{On $s$-extremal singly even self-dual 
$[24k+8,12k+4,4k+2]$ codes\footnote{Keywords: 
extremal doubly even self-dual code, 
$s$-extremal singly even self-dual code,
covering radius.
Mathematics Subject Classification: 94B05, 94B75.}
}

\author{
Masaaki Harada\thanks{
Research Center for Pure and Applied Mathematics,
Graduate School of Information Sciences,
Tohoku University, Sendai 980--8579, Japan.}
and 
Akihiro Munemasa\thanks{
Research Center for Pure and Applied Mathematics,
Graduate School of Information Sciences,
Tohoku University,
Sendai 980--8579, Japan.}
}

\maketitle

\begin{abstract}
A relationship between $s$-extremal singly even self-dual 
$[24k+8,12k+4,4k+2]$ codes and extremal doubly even self-dual 
$[24k+8,12k+4,4k+4]$ codes 
% with covering radius $4k+2$, 
with covering radius meeting the Delsarte bound,
is established.
As an example of the relationship, 
$s$-extremal singly even self-dual $[56,28,10]$ codes are constructed
for the first time.
% The existence of such self-dual codes was not previously known.
In addition, 
we show that there is no extremal doubly even self-dual code of 
length $24k+8$ with covering radius meeting the 
Delsarte bound for $k \ge 137$.
Similarly, we show that there is no extremal doubly even self-dual
code of length $24k+16$ with covering radius meeting 
the Delsarte bound for $k \ge 148$.
We also determine the covering radii of some 
extremal doubly even self-dual codes of length $80$.
\end{abstract}

%%%%%%%%%%%%%%%%%%%%%%%%%%%%%%%%%%
\section{Introduction}
A (binary) $[n,k]$ {\em code} $C$ is a $k$-dimensional vector subspace
of $\FF_2^n$,
where $\FF_2$ denotes the finite field of order $2$.
All codes in this paper are binary.
The parameter $n$ is called the {\em length} of $C$.
The {\em weight} $\wt(x)$ of a vector $x \in \FF_2^n$ is
the number of non-zero components of $x$.
A vector of $C$ is called a {\em codeword} of $C$.
The minimum non-zero weight of all codewords in $C$ is called
the {\em minimum weight} of $C$ and an $[n,k]$ code with minimum
weight $d$ is called an $[n,k,d]$ code.
The {\em automorphism group} of $C$ consists of all permutations
of the coordinates of $C$ which preserve $C$.
Two codes are {\em equivalent} if one can be
obtained from the other by permuting the coordinates.
% An {\em automorphism} of $C$ is a permutation of the coordinates of $C$
% which preserves $C$.
% The set consisting of all automorphisms of $C$ is called the
% {\em automorphism group} of $C$.
% and it is denoted by $\Aut(C)$.

The {\em dual} code $C^{\perp}$ of a code
$C$ of length $n$ is defined as
$
C^{\perp}=
\{x \in \FF_2^n \mid x \cdot y = 0 \text{ for all } y \in C\},
$
where $x \cdot y$ is the standard inner product.
A code $C$ is called {\em self-dual}
if $C=C^\perp$. 
The minimum weight $d(C)$ of a self-dual code $C$ of length
$n$ is bounded by $d(C)\le 4 \lfloor n/24 \rfloor +4$ unless $n
\equiv 22 \pmod{24}$ when $d(C) \le 4 \lfloor n/24 \rfloor+6$~\cite{Rains}.
We say that a self-dual
code meeting the upper bound is {\em extremal}.

Self-dual codes are divided into two classes, namely,
doubly even self-dual codes and singly even self-dual codes.
A self-dual code is called
{\em doubly even} and {\em singly even}
if all codewords have weight
$\equiv 0 \pmod 4$ and if some codeword has weight
$\equiv 2 \pmod 4$, respectively.
A doubly even self-dual code of length $n$ exists if and
only if $n \equiv 0 \pmod 8$, while a singly even self-dual code
of length $n$ exists if and
only if $n$ is even.

%%%%%%%%%%%%%%
Let $C$ be a singly even self-dual code and let $C_0$ denote the
subcode of codewords having weight $\equiv0\pmod4$. Then $C_0$ is
a subcode of codimension $1$. The {\em shadow} $S$ of $C$ is
defined to be $C_0^\perp \setminus C$. 
% There are cosets
% $C_1,C_2,C_3$ of $C_0$ such that $C_0^\perp = C_0 \cup C_1 \cup
% C_2 \cup C_3 $, where $C = C_0  \cup C_2$ and $S = C_1 \cup C_3$.
Shadows were introduced by Conway and Sloane~\cite{C-S}, in order
to provide restrictions on the weight enumerators of singly even
self-dual codes. 
%%%%
The minimum weight $d(D)$ of a coset $D$ is the minimum non-zero weight of
all vectors in the coset.
Let $C$ and $S$ be a singly even 
self-dual code of length $n$ and its shadow, respectively.
Bachoc and Gaborit~\cite{BG} showed that 
\begin{equation}\label{eq:BG}
d(S) \le \frac{n}{2}+4-2d(C),
\end{equation}
unless $n \equiv 22 \pmod{24}$ and $d(C)= 4 \lfloor n/24 \rfloor +6$,
in which case $d(S) = \frac{n}{2}+8-2d(C)$.
A singly even self-dual code $C$ meeting the bound
is called {\em $s$-extremal}.
The existence of many $s$-extremal singly even self-dual codes 
was mentioned in~\cite{BG}.
For minimum weight $10$, 
it is known that there is an 
$s$-extremal singly even self-dual code
of lengths $46,50,52,54,58$ (see~\cite[Section 7]{BG}).

The {\em covering radius} $R(C)$ of a code $C$ of length $n$
is the smallest integer $R$ such that spheres of radius $R$
around codewords of $C$ cover the space $\FF_2^n$.
It is known that
%% the covering radius $R(C)$ is the same as the weight of the
%% coset of largest weight.
the covering radius $R(C)$ is the same as the largest value
among the minimum weights of nontrivial cosets 
(see~\cite[Theorem~13 (i)]{Pless}).
% Here the weight of a coset is the smallest weight of
% a vector in the coset.
% We also denote the weight of a coset $E$ by $d(E)$.
The covering radius is a basic and important geometric parameter of
a code.
%%%%%%%
Let $C$ be a code of length $n$ and 
let $A^\perp_i$ be the number of vectors of weight $i$ in $C^\perp$.
It was shown in~\cite{D} that
$R(C) \le |\{i \mid A^\perp_i \ne 0, 0 < i \le n\}|$.
The bound is called the Delsarte bound (see~\cite{A-P}).
Much work has been done concerning covering radii of 
extremal doubly even self-dual codes
(see e.g., \cite{A-P}, \cite{HM07}, \cite{Ozeki}, \cite{Tsai-D11}).

In this paper, 
a relationship between $s$-extremal singly even self-dual 
$[24k+8,12k+4,4k+2]$ codes and extremal doubly even self-dual 
$[24k+8,12k+4,4k+4]$ codes 
% with covering radius $4k+2$, 
with covering radius meeting the Delsarte bound,
is established.
As an example of the relationship, 
we actually construct $s$-extremal singly even
self-dual $[56,28,10]$ codes from some
extremal doubly even self-dual $[56,28,12]$ 
codes with covering radius $10$.
The existence of such self-dual codes was not previously known.
We also determine the covering radii of some known
extremal doubly even self-dual codes of length $80$, along with
two new extremal doubly even self-dual codes of the same length.
It is shown that there is no
$s$-extremal singly even self-dual $[24k+8,12k+4,4k+2]$ code
for $k \ge 137$.
By the above relationship, 
there is no extremal doubly even self-dual code of length
$24k+8$ with covering radius meeting the 
Delsarte bound for $k \ge 137$.
Finally, we show that there is no $s$-extremal singly even 
self-dual $[24k+16,12k+8,4k+4]$ code for $k \ge 148$.
By Proposition~4 in~\cite{HM07},
there is no extremal doubly even self-dual code
of length $24k+16$ with covering radius meeting the Delsarte bound
for $k \ge 148$.

All computer calculations in this paper
were done by {\sc Magma}~\cite{Magma}.

%%%%%%%%%%%%%%%%%%%%%%%%%%%%%%%%%%%%%%%%%%%
\section{A relationship between $s$-extremal self-dual 
$[24k+8,12k+4,4k+2]$ codes and certain extremal doubly even self-dual codes}

In this section, a relationship between $s$-extremal singly even self-dual 
$[24k+8,12k+4,4k+2]$ codes and extremal doubly even self-dual 
$[24k+8,12k+4,4k+4]$ codes with covering radius $4k+2$, 
is established.

Let $C$ be a singly even self-dual code and let $C_0$ denote the
subcode of codewords having weight $\equiv0\pmod4$. 
Then there are cosets
$C_1,C_2,C_3$ of $C_0$ such that $C_0^\perp = C_0 \cup C_1 \cup
C_2 \cup C_3 $, where $C = C_0  \cup C_2$ and $S = C_1 \cup C_3$.
Two self-dual codes $C$ and $C'$ of length $n$
are said to be {\em neighbors} if $\dim(C \cap C')=n/2-1$. 
If $C$ is a singly even
self-dual code of length divisible by $8$, then $C$ has two doubly
even self-dual neighbors, namely $C_0 \cup C_1$ and $C_0 \cup C_3$
(see~\cite{BP}).

% Much work has been done concerning covering radii.
% Assmus and Pless~\cite{A-P} studied the covering radii of extremal
% doubly even self-dual codes.
% For example, 
% it is known that  $8 \le R_{56} \le 10$,
% where $R_{56}$ is the covering radius of an extremal doubly
% even self-dual code of length $56$~\cite[Table III]{A-P}.

\begin{thm}\label{thm}
If $C$ is an $s$-extremal singly even self-dual $[24k+8,12k+4,4k+2]$ code,
then the two doubly even self-dual neighbors  $C_0 \cup C_1$ and $C_0 \cup C_3$
are extremal and their covering radii are $4k+2$, meeting the Delsarte bound.
Conversely, if $D$ is an extremal doubly even self-dual
$[24k+8,12k+4,4k+4]$ code
with covering radius $4k+2$,
then $D$ has an
$s$-extremal singly even self-dual $[24k+8,12k+4,4k+2]$ neighbor.
\end{thm}
\begin{proof}
Suppose that $C$ is an $s$-extremal singly even self-dual $[24k+8,12k+4,4k+2]$ code.
Since $d(C_i) \ge 4k+4$ $(i=0,1,3)$, 
$C_0 \cup C_1$ and $C_0 \cup C_3$ have minimum weight at least $4k+4$.
By the bound on the minimum weight, 
both $C_0 \cup C_1$ and $C_0 \cup C_3$ have minimum weight $4k+4$,
that is, the codes are extremal.
The sets $C_2 \cup C_3$ and $C_1 \cup C_2$ are cosets of 
$C_0 \cup C_1$ and $C_0 \cup C_3$, respectively.
Since $d(C_2)=4k+2$, $d(C_2 \cup C_3) = d(C_1 \cup C_2) = 4k+2$.
Hence, both $C_0 \cup C_1$ and $C_0 \cup C_3$ have covering radius
at least $4k+2$.
By the Delsarte bound, 
both $C_0 \cup C_1$ and $C_0 \cup C_3$ have covering radius $4k+2$.

Conversely, suppose that $D$ is an extremal doubly even self-dual
$[24k+8,12k+4,4k+4]$ code with covering radius $4k+2$.
Then there is a coset $v+D$ of minimum weight $4k+2$.
Set $D_0=\{x \in D \mid v \cdot x =0\}$, which is a subcode of index $2$.
Then there are cosets
$D_1,D_2,D_3$ of $D_0$ such that $D_0^\perp = D_0 \cup D_1 \cup
D_2 \cup D_3 $, where $D= D_0  \cup D_2$, $v+D=D_1  \cup D_3$,
$D_1$ consists of the vectors whose weight is congruent to $2 \pmod 4$ and
$D_3$ consists of the vectors whose weight is congruent to
$0 \pmod 4$~\cite[Lemma~3]{BP}.
In addition, $D_0  \cup D_1$ is a singly even self-dual code with shadow
$D_2  \cup D_3$.
Since $d(D_1)  = 4k+2$,
$d(D_2) \ge 4k+4$ and $d(D_3) \ge 4k+4$,
$D_0  \cup D_1$ is a singly even self-dual $[24k+8,12k+4,4k+2]$ code 
with shadow of minimum weight at least $4k+4$.
% Of course, $D_0  \cup D_1$ is a neighbor of $D$.
By the upper bound~\eqref{eq:BG},
the minimum weight of the shadow is exactly $4k+4$.
Of course, $D_0  \cup D_1$ is a neighbor of $D$.
\end{proof}

\begin{rem}\label{rem}
A similar result can be found in~\cite[Proposition~4]{HM07} for 
$s$-extremal singly even self-dual $[24k+16,12k+8,4k+4]$ codes
(see also Lemma~\ref{lem:24k+16}).
\end{rem}

Let $A_i$ and $B_i$ denote the number of codewords of
weight $i$ in a singly even self-dual code $C$ of length $n$ and 
the number of vectors of weight $i$ in its shadow $S$, respectively.
The weight enumerators of $C$ and $S$ are given by
$\sum_{i=0}^{n} A_i y^i$ and $\sum_{i=0}^{n} B_i y^i$, 
respectively.
The weight enumerators of an $s$-extremal singly even self-dual 
$[n,n/2,d]$ code and its shadow
are uniquely determined~\cite{BG}.

\begin{cor}\label{cor}
If $C$ is an $s$-extremal singly even self-dual $[24k+8,12k+4,4k+2]$ code,
then the weight enumerator of a coset of minimum weight $4k+2$
in the two doubly even self-dual neighbors 
$C_0 \cup C_1$ and $C_0 \cup C_3$ is given by
\begin{equation}\label{eq:W}
\sum_{i \equiv 2 \pmod 4} A_iy^i 
+ \sum_{i} \frac{B_i}{2} y^i.
\end{equation}
\end{cor}
\begin{proof}
By Theorem~\ref{thm},
$C_2 \cup C_3$ (resp.\ $C_2 \cup C_1$) is a coset of minimum weight
$4k+2$ in $C_0 \cup C_1$ (resp.\ $C_0 \cup C_3$).
The weight enumerator of an extremal doubly even self-dual code
of length $n$ is uniquely determined.
Since $C_0 \cup C_1$ and $C_0 \cup C_3$ are extremal doubly even 
self-dual codes, $C_1$ and $C_3$ have
the identical weight enumerators.
Hence, the weight enumerators of $C_2 \cup C_3$ and $C_2 \cup C_1$ 
are given by~\eqref{eq:W}.
By Corollary~1 to Theorem~1 in~\cite{A-P}, 
the cosets of minimum weight $4k+2$, if there are any, 
have unique weight enumerators.
\end{proof}

As an example of Theorem~\ref{thm}, we consider the case $k=1$.
There are five inequivalent 
extremal doubly even self-dual codes of length $32$~\cite{CPS}.
These codes are denoted by C81, C82, C83, C84 and C85 
in~\cite[Table~A]{CPS}.
Every extremal doubly even self-dual code of length $32$
has covering radius $6$~\cite[Theorem~3]{A-P}.
These codes contain
$11253,14756,12236,11354$ and $11321$ cosets of minimum
weight $6$, respectively~\cite[Table~II]{A-P} and 
\cite[Tables~1, 2, 6, 8]{CCM}.
By Theorem~\ref{thm},
$11253,14756,12236,11354$ and $11321$ $s$-extremal singly even
self-dual $[32,16,6]$ codes are obtained, respectively.
We verified by {\sc Magma}~\cite{Magma} 
% by the {\sc Magma} function {\tt IsIsomorphic}~\cite{Magma} 
that the $11253,14756,12236,11354$ and $11321$ codes are divided into 
$6, 2, 3, 7$ and $10$ inequivalent codes,
respectively.
By determining equivalence of the $28$ codes,
there are $19$ inequivalent
$s$-extremal singly even self-dual $[32,16,6]$ codes. 
This classification coincides with the result 
listed in~\cite[p.~29]{BG}.

% [ 1, 11253 ]
% [ 2, 14756 ]
% [ 3, 12236 ]
% [ 4, 11354 ]
% [ 5, 11321 ]
% [ 6, 2, 3, 7, 10 ]
%28
%inequivalent= 19
%[ 1, 4, 2, 1 ]
%[ 1, 4, 4, 3 ]
%[ 1, 5, 6, 3 ]
%[ 2, 3, 1, 3 ]
%[ 2, 5, 2, 2 ]
%[ 3, 4, 1, 6 ]
%[ 3, 5, 2, 4 ]
%[ 4, 5, 2, 10 ]
%[ 4, 5, 7, 1 ]

%%%%%%%%%%%%%%%%%%%%%%%%%%%%%%%%%%%
\section{Construction of $s$-extremal singly even 
self-dual $[56,28,10]$ codes}\label{sec:56}

In this section, we consider the case $k=2$.
Some $s$-extremal singly even self-dual $[56,28,10]$ codes
are constructed.  
The existence of such self-dual codes was not previously known
(see~\cite{BG}).

% Let $A_i$ and $B_i$ denote the number of codewords of
% weight $i$ in a singly even self-dual code $C$ of length $n$ and 
% the number of vectors of weight $i$ in its shadow $S$, respectively.
% The weight enumerators of $C$ and $S$ are given by
% $\sum_{i=0}^{n} A_i y^i$ and $\sum_{i=0}^{n} B_i y^i$, 
% respectively.
% 
By using~\cite[(10), (11)]{C-S},
the weight enumerators of an $s$-extremal singly even self-dual
$[56,28,10]$ code and its shadow
are determined as follows:
\begin{align*}
&1
+ 308 y^{10}
+ 3990 y^{12}
+ 42900 y^{14}
+ 311850 y^{16}
+ 1583120 y^{18}
+ 5847688 y^{20}
\\ &
+ 15961680 y^{22}
+ 32458965 y^{24}
+ 49520856 y^{26}
+ 56972740 y^{28}
+ \cdots,
\\ &
8400 y^{12}
+ 620928 y^{16}
+ 11704000 y^{20}
+ 64901760 y^{24}
+ 113965280 y^{28}
+ \cdots,
\end{align*}
respectively.
% This was done by {\sc Magma}~\cite{Magma}.

The existence of extremal doubly even self-dual codes of length $56$
with covering radius $10$ is known~\cite{HM07}, \cite{Ozeki} and 
\cite{Tsai-D11}.
Theorem~\ref{thm} gives the existence of an $s$-extremal singly even
self-dual $[56,28,10]$ code.
As an example, we actually construct $s$-extremal singly even
self-dual $[56,28,10]$ codes from 
the extremal double circulant doubly even self-dual $[56,28,12]$ 
code $C_{56,1}$ in~\cite{HGK-DCC} and 
the extremal doubly even self-dual $[56,28,12]$ code in~\cite{BYB}, 
which have covering radius $10$.
We denote the code in~\cite{BYB} by $BYB$.
Ozeki~\cite{Ozeki} determined the complete coset weight distributions of 
$C_{56,1}$ and $BYB$.
The codes contain $2925$ and $6552$  cosets of minimum weight 
$10$, respectively~\cite{Ozeki}.
By Theorem~\ref{thm},
$2925$ and $6552$ $s$-extremal singly even
self-dual $[56,28,10]$ codes are obtained, respectively.
%%%%%%%%%%%%%%%%%%%%
We verified by {\sc Magma}~\cite{Magma} 
% by the {\sc Magma} function {\tt IsIsomorphic}~\cite{Magma} 
that the $2925$ codes are divided into $70$ inequivalent codes.
We denote the new codes by $D_1,D_2,\ldots,D_{70}$.
These codes are constructed as 
$(C_{56,1})_0  \cup ( v+(C_{56,1})_0 )$, where
$(C_{56,1})_0=\{x \in C_{56,1} \mid v \cdot x =0\}$
and the supports $\supp(v)$ of $v$ are listed in Table~\ref{Tab:56}.
The orders $|\Aut|$ of the automorphism groups 
calculated by {\sc Magma}~\cite{Magma} 
% the {\sc Magma} function {\tt AutomorphismGroup}~\cite{Magma} 
are  listed in Table~\ref{Tab:56Aut}.
%%%%%%%%%%%%%%%%%%%%
We verified by {\sc Magma}~\cite{Magma} that the automorphism group 
of the code $BYB$ acts on the $6552$ cosets of minimum weight 
$10$ transitively.
% by the {\sc Magma} function {\tt IsIsomorphic}~\cite{Magma} 
This means that the $6552$ $s$-extremal singly even
self-dual $[56,28,10]$ codes are equivalent.
The unique $s$-extremal code $E$ is constructed as 
$(BYB)_0  \cup ( v+(BYB)_0 )$, where
$(BYB)_0=\{x \in BYB \mid v \cdot x =0\}$
and 
\[
\supp(v)=\{1, 2, 3, 17, 21, 22, 28, 37, 51, 55\}.
\]
We verified by {\sc Magma}~\cite{Magma} 
% the {\sc Magma} function 
% {\tt AutomorphismGroup}~\cite{Magma} 
that the code $E$ has automorphism group of order $9$.

%%%%%%%%%%%%%%%%%%%%%%%%%%%%%%%
\begin{table}[thbp]
\caption{New $s$-extremal singly even self-dual $[56,28,10]$ codes}
\label{Tab:56}
\begin{center}
%{\small
{\footnotesize
%{\scriptsize
%\begin{tabular}{c|c|l|ccccccc}
\begin{tabular}{c|l|c|l}
\noalign{\hrule height0.8pt}
$i$ &  \multicolumn{1}{c|}{$\supp(v)$}  & 
$i$ &  \multicolumn{1}{c}{$\supp(v)$} \\
\hline
 1&$\{1,2,3,4,5,9,10,24,54,55\}$ &
 2&$\{1,2,3,6,9,10,21,23,36,45\}$ \\
 3&$\{1,2,3,7,11,16,32,33,46,51\}$ &
 4&$\{1,2,3,7,12,13,14,31,37,46\}$ \\
 5&$\{1,2,3,4,5,15,49,50,53,55\}$ &
 6&$\{1,2,3,19,24,26,39,40,44,48\}$ \\
 7&$\{1,2,3,18,25,36,38,39,42,53\}$ &
 8&$\{1,2,3,6,15,27,36,40,45,46\}$ \\
 9&$\{1,2,3,6,27,42,43,46,55,56\}$ &
10&$\{1,2,3,9,11,33,40,42,47,50\}$ \\
11&$\{1,2,4,6,10,12,24,31,46,52\}$ &
12&$\{1,2,3,7,15,25,30,33,43,51\}$ \\
13&$\{1,2,3,22,23,25,31,41,45,53\}$ &
14&$\{1,2,3,14,20,28,30,35,41,53\}$ \\
15&$\{1,2,4,7,17,23,28,42,54,55\}$ &
16&$\{1,2,3,20,24,26,29,34,38,50\}$ \\
17&$\{1,2,4,10,11,25,27,30,37,52\}$ &
18&$\{1,2,4,7,18,20,25,30,32,36\}$ \\
19&$\{1,2,3,12,16,21,24,30,31,48\}$ &
20&$\{1,2,3,12,16,24,29,35,36,41\}$ \\
21&$\{1,2,3,12,19,24,30,34,39,56\}$ &
22&$\{1,2,3,5,10,25,30,36,45,55\}$ \\
23&$\{1,2,3,6,15,20,23,27,33,42\}$ &
24&$\{1,2,3,23,25,29,31,39,47,53\}$ \\
25&$\{1,2,3,18,21,24,44,47,52,53\}$ &
26&$\{1,2,3,4,9,27,29,32,36,37\}$ \\
27&$\{1,2,3,7,15,18,24,31,37,52\}$ &
28&$\{1,2,3,18,22,26,31,34,41,43\}$ \\
29&$\{1,2,3,4,6,16,42,45,50,52\}$ &
30&$\{1,2,3,4,11,26,27,29,32,47\}$ \\
31&$\{1,2,3,4,11,20,25,38,48,54\}$ &
32&$\{1,2,3,12,14,15,18,23,36,54\}$ \\
33&$\{1,2,3,4,24,28,31,39,47,52\}$ &
34&$\{1,2,3,4,14,19,35,39,48,49\}$ \\
35&$\{1,2,3,7,17,25,41,50,55,56\}$ &
36&$\{1,2,3,5,7,13,20,34,48,52\}$ \\
37&$\{1,2,3,7,15,25,38,40,42,43\}$ &
38&$\{1,2,3,4,9,20,46,48,53,54\}$ \\
39&$\{1,2,3,9,11,19,24,30,32,45\}$ &
40&$\{1,2,3,4,13,26,39,48,52,55\}$ \\
41&$\{1,2,3,4,11,17,28,30,46,55\}$ &
42&$\{1,2,3,8,13,20,38,41,44,56\}$ \\
43&$\{1,2,4,5,28,33,34,47,51,55\}$ &
44&$\{1,2,3,23,26,29,34,37,40,49\}$ \\
45&$\{1,2,3,6,12,22,28,31,41,51\}$ &
46&$\{1,2,4,7,18,20,26,29,43,53\}$ \\
47&$\{1,2,3,21,32,38,42,48,52,53\}$ &
48&$\{1,2,3,10,11,12,14,17,20,50\}$ \\
49&$\{1,2,3,4,6,15,19,36,45,55\}$ &
50&$\{1,2,3,5,8,11,32,34,37,42\}$ \\
51&$\{1,2,3,5,9,27,32,45,46,56\}$ &
52&$\{1,2,4,6,11,13,15,16,37,47\}$ \\
53&$\{1,2,3,4,8,16,17,19,37,55\}$ &
54&$\{1,2,3,18,23,34,36,44,47,53\}$ \\
55&$\{1,2,3,5,6,21,25,32,33,49\}$ &
56&$\{1,2,3,23,35,37,40,45,46,53\}$ \\
57&$\{1,2,3,4,24,28,37,41,45,48\}$ &
58&$\{1,2,4,7,23,29,32,37,46,50\}$ \\
59&$\{1,2,4,11,16,22,43,46,52,53\}$ &
60&$\{1,2,3,6,12,14,20,37,46,47\}$ \\
61&$\{1,2,4,5,22,29,36,46,49,54\}$ &
62&$\{1,2,3,22,23,24,35,40,50,53\}$ \\
63&$\{1,2,3,5,12,14,26,32,42,48\}$ &
64&$\{1,2,3,7,15,23,31,38,39,42\}$ \\
65&$\{1,2,4,12,27,28,31,34,42,53\}$ &
66&$\{1,2,3,7,8,14,34,36,42,52\}$ \\
67&$\{1,2,3,4,28,29,32,37,38,46\}$ &
68&$\{1,2,3,17,24,34,45,46,48,49\}$ \\
69&$\{1,2,3,8,9,12,14,20,42,46\}$ &
70&$\{1,2,3,4,13,18,36,37,52,55\}$ \\
\noalign{\hrule height0.8pt}
\end{tabular}
}
\end{center}
\end{table}
%%%%%%%%%%%%%%%%%%%%%%%%%%%%%%%

%%%%%%%%%%%%%%%%%%%%%%%%%%%%%%%
\begin{table}[thb]
\caption{Automorphism groups of $D_i$ ($i=1,2,\ldots,70$)}
\label{Tab:56Aut}
\begin{center}
%{\small
{\footnotesize
%{\scriptsize
%\begin{tabular}{c|c|l|ccccccc}
\begin{tabular}{c|l}
\noalign{\hrule height0.8pt}
$|\Aut|$ &  \multicolumn{1}{c}{$i$} \\
\hline
1 &  
2, 3, 4, 7, 8, 9, 10, 11, 12, 14, 15, 16, 17, 18, 19, 20, 21, 22, 23, 25, 27, \\
& 31, 34, 37, 41, 43, 44, 45, 47, 48, 49, 55, 57, 58, 61, 63, 67, 68, 70 \\
\hline
2 &
1, 5, 6, 13, 24, 26, 28, 29, 30, 32, 33, 35, 36, 39, 40, 42, 46, 50, 51, 52, \\
& 53, 54, 56, 59, 60, 62, 64, 65, 66, 69 \\
\hline
6 &  38 \\
\noalign{\hrule height0.8pt}
\end{tabular}
}
\end{center}
\end{table}
%%%%%%%%%%%%%%%%%%%%%%%%%%%%%%%

%%%%%%%%%%%%%%%%%%%%
By comparing the automorphism groups of $D_i$ $(i=1,2,\ldots,70)$
and $E$, 
we see that these $71$ codes are pairwise inequivalent.
Therefore, we have the following:

\begin{prop}
There are at least $71$ $s$-extremal
singly even self-dual $[56,28,10]$ codes.
\end{prop}

The existence of such self-dual codes was not previously 
known (see~\cite{BG}).

\section{Covering radii of some extremal doubly even 
self-dual codes of length 80}

In this section, we consider the case $k=3$.
The covering radii of some extremal doubly even self-dual codes of
length $80$ are determined.

The extended quadratic residue 
code $QR_{80}$ of length $80$ is an extremal doubly even self-dual code.
It was shown in~\cite{DH} that there are $11$ extremal
doubly even self-dual $[80,40,16]$ codes with an automorphism of
order 19, up to equivalence.  
The $11$ codes are denoted by $C_{80,1}, C_{80,2},\ldots,C_{80,11}$
in~\cite{DH}.
Also, none of them is equivalent to $QR_{80}$.
% The code $C_{80,7}$ has the largest automorphism group among
% the $11$ codes.
There are six (resp.\ five) inequivalent pure (resp.\ bordered)
double circulant extremal doubly even self-dual codes of 
length $80$~\cite{GH-DCC88}.
These codes are denoted by $P_{80,i}$ $(i=1,2,\ldots,6)$
(resp.\ $B_{80,i}$ $(i=1,2,\ldots,5)$) in~\cite{GH-DCC88}.
Note that $P_{80,1}$ and $B_{80,5}$ are equivalent.
Also, we verified by {\sc Magma}~\cite{Magma} 
% by the {\sc Magma} function {\tt IsIsomorphic}~\cite{Magma} 
that $P_{80,1}$ is equivalent to $QR_{80}$.
Very recently, $14$ more inequivalent extremal doubly even
self-dual codes of length $80$ have been found in \cite{KY16}.
These codes are denoted by 
$\LL_{80,i}$ $(i=1,2,\ldots,14)$ in~\cite[Table~11]{KY16}.
Hence, at least $35$ inequivalent extremal doubly even
self-dual codes of length $80$ exist\footnote{The number of
known extremal doubly even self-dual codes of length $80$
is incorrectly reported as $36$ in \cite[Theorem~6.1]{KY16}.}.

%%%%%%%%%%%%%%%%%%%%%%%%
By the method given in~\cite{HK}, we found two
extremal doubly even self-dual codes
$N_{80,1}$ and $N_{80,2}$ of length $80$.
The code $N_{80,i}$ $(i=1,2)$
has generator matrix $(\ I_{40}\ \ M_{80,i}\ )$, 
where $I_{40}$ denotes the identity matrix of order $40$ and 
$M_{80,1}$ (resp.\ $M_{80,2}$) is listed in 
Figure~\ref{Fig:801} (resp.\ Figure~\ref{Fig:802}).
In the figures,
each row of $M_{80,i}$ is written in octal using $0=(000)$,
$1=(001),\ldots,6=(110)$ and $7=(111)$, together with 
$a=(0)$ and $b=(1)$, where, for example, 
the $(1,2)$-entry indicates the second row and 
the $(2,1)$-entry indicates the 5th row.
We verified by {\sc Magma}~\cite{Magma} 
% by the {\sc Magma} function {\tt AutomorphismGroup}~\cite{Magma} 
that $N_{80,1}$  and $N_{80,2}$ 
have automorphism groups of orders $32$ and $64$, respectively.
Since none of the $35$ known 
inequivalent extremal doubly even self-dual codes of length $80$
has automorphism group of order $32,64$,
we have the following:

\begin{prop}
There are at least $37$ inequivalent extremal doubly even self-dual 
codes of length $80$. 
\end{prop}

By the Assmus--Mattson theorem, 
the supports of the codewords of minimum weight in an extremal
doubly even self-dual code of length $80$ form a
self-orthogonal $3$-$(80,16,665)$ design.
We verified that each of the $37$ extremal doubly even self-dual 
codes of length $80$ is generated by the codewords
of minimum weight.
Hence, we have the following:

\begin{cor}
There are at least $37$ non-isomorphic self-orthogonal
$3$-$(80,16,665)$ designs having $2$-rank $40$.
\end{cor}

%%%%%%%%%%%%%%%%%%%%%%%%
% Hence, 
% $21$ inequivalent extremal doubly even self-dual codes of length $80$
% are known.
% Currently, the covering radii of the known 
% extremal doubly even self-dual codes of length $80$
% are not determined.
% 
% It is not known whether there is an extremal
% doubly even self-dual code of length $80$ with covering
% radius $14$.

%%%%%%%%%%%%%%%%%  Fig  %%%%%%%%%%%%%%%%%
\begin{figure}[thb]
\centering
%{\scriptsize
{\footnotesize
%{\small
0040734771561a
4434543163675b
2602474666733b
1715423124350b\\
0352004645161a
3206160642302a
2660112641633b
6053127640067a\\
7431246537036b
4577441737365b
2663035540577b
1725203477272b\\
0346314420530a
4577753007251a
1114607163256a
4052516666522a\\
1346325653523b
7600030245123a
0463176442323a
4625262036154b\\
2706324200063a
5757747317034b
2373176750413a
5561262173200b\\
5513453515272b
2251030451530a
2247576744126b
6240715602721b\\
4243224261422a
6535307327614a
7642356744703a
7335772175344b\\
3142160261567a
5075665327676b
2022127744732a
2372131202427b\\
6216136061561a
0264135550102a
4526643473044a
1130243355350b\\
}
\caption{Matrix $M_{80,1}$ (in octal)}
\label{Fig:801}
\end{figure}
%%%%%%%%%%%%%%%%%  Fig  %%%%%%%%%%%%%%%%%

%%%%%%%%%%%%%%%%%  Fig  %%%%%%%%%%%%%%%%%
\begin{figure}[thb]
\centering
%{\scriptsize
{\footnotesize
%{\small
2017612646116b
1314312723575b
7122225026433b
4035272764050b\\
5472655005261a
6526731002002a
7140743001533b
4004001777410b\\
5466360400441a
6520567600712a
7143664300677b
4005452237172b\\
5466545260630a
1257102647151a
3143721054621b
6005430751155b\\
3311203764154a
2120661405223a
5343727202023a
6672344101523a\\
0751202337414b
0077116557334b
7453727110713a
7536344044677a\\
0233202355172b
0206116566147b
7567327104226b
3560144042421b\\
6214302356055b
3215556567514a
2162507104403a
2415123735044b\\
6662731421667a
7022743210201a
0075001673345b
7452760442727b\\
4241010156116b
2233013467575b
1206012233344a
4610412515050b\\
}
\caption{Matrix $M_{80,2}$ (in octal)}
\label{Fig:802}
\end{figure}
%%%%%%%%%%%%%%%%%  Fig  %%%%%%%%%%%%%%%%%

We determine the covering radii of the codes 
$C_{80,i}$ $(i=1,2,\ldots,11)$ in~\cite{DH}, 
$P_{80,i}$ $(i=1,2,\ldots,6)$,
$B_{80,i}$ $(i=1,2,3,4)$ in~\cite{GH-DCC88},
$\LL_{80,i}$ $(i=1,2,\ldots,14)$ in~\cite{KY16}
and the new codes $N_{80,i}$ $(i=1,2)$ as follows.
% Due to computer time limitations, we have only been able to
% accomplish our search in extremal doubly even self-dual 
% codes with relatively large automorphism groups.
By the Delsarte bound, the covering radius of an
extremal doubly even self-dual code of length $80$
is at most $14$.
By modifying the method in~\cite{Munemasa}, we verified that
there is no coset of minimum weight $13$ and 
there is a coset of minimum weight $12$ for the codes
% $C_{80,7}$, 
$C_{80,i}$ $(i=1,2,\ldots,11)$, 
$P_{80,i}$ $(i=1,2,\ldots,6)$, $B_{80,i}$ $(i=1,2,3)$,
$\LL_{80,i}$ $(i=1,2,\ldots,14)$
and $N_{80,i}$ $(i=1,2)$.
%%%%%%
In this calculation, the automorphism groups are useful
to reduce the number of cosets which must be considered.
%%%%%%
Also, we verified that
there is no coset of minimum weight $14$ and
there is a coset of minimum weight $13$ for the code $B_{80,4}$.
Hence, we have the following:

\begin{prop}
The covering radii of the codes 
$C_{80,i}$ $(i=1,2,\ldots,11)$, 
$P_{80,i}$ $(i=1,2,\ldots,6)$, $B_{80,i}$ $(i=1,2,3)$,
$\LL_{80,i}$ $(i=1,2,\ldots,14)$
and 
$N_{80,i}$ $(i=1,2)$ are $12$.
The covering radius of the code $B_{80,4}$
is $13$.
\end{prop}

A vector $a$ of a coset $U$ is called a coset leader of $U$ if
the weight of $a$ is minimal in $U$.
We give the support of a coset leader of a
coset of minimum weight $13$ in $B_{80,4}$:
\[
\{ 2, 5, 8, 11, 14, 17, 20, 23, 26, 29, 32, 35, 38\}.
\]

The covering radii of extremal doubly even self-dual codes of
length $80$ are determined for the first time.
The above proposition indicates that 
there are many extremal doubly even self-dual codes with
covering radius not meeting the Delsarte bound.
It is an open problem to determine whether there is
an extremal doubly even self-dual $[80,40,16]$ code with
covering radius meeting the Delsarte bound or not. 
Equivalently, it is an open problem to determine whether
there is an $s$-extremal
singly even self-dual $[80,40,14]$ code or not.

By Corollary~2 to Theorem~1 in~\cite{A-P}, 
the cosets of minimum weights $13$ and $14$, if there are any, 
have unique weight enumerators.
We verified by {\sc Magma}~\cite{Magma} 
% by the {\sc Magma} 
% function {\tt WeightDistribution}~\cite{Magma} 
that the unique weight enumerator of a coset of minimum weight $13$ 
in $B_{80,4}$ is as follows:
\begin{align*}
&
          560y^{13}
+        12624y^{15}
+       184080y^{17}
+      2103920y^{19}
+     18380960y^{21}
\\ &
+    124273760y^{23}
+    661003616y^{25}
+   2796701600y^{27}
+   9492488720y^{29}
\\ &
+  26027430000y^{31}
+  57969919280y^{33}
+ 105320725328y^{35}
\\ &
+ 156558001600y^{37}
+ 190784587840y^{39}
+ \cdots.
\end{align*}
% 
% By using~\cite[(10), (11)]{C-S},
By using~\cite[(10), (11)]{C-S},
the weight enumerators of an $s$-extremal singly even self-dual
$[80,40,14]$ code and its shadow
are determined as follows:
\begin{align*}
&
1 
+ 3200  y^{14} 
+ 47645  y^{16} 
+ 640640  y^{18} 
+ 6452992  y^{20} 
+ 49304320  y^{22} 
\\ &
+ 294979360  y^{24} 
+ 1398270720  y^{26} 
+ 5294263040  y^{28} 
+ 16137190784  y^{30} 
\\ &
+ 39853463650  y^{32} 
+ 80135036800  y^{34} 
+ 131652451840  y^{36} 
\\ &
+ 177157460480  y^{38} 
+ 195552496832  y^{40} 
+ \cdots,
\\&
99840  y^{16} 
+ 12859392  y^{20} 
+ 590187520  y^{24} 
+ 10587822080  y^{28} 
\\ &
+ 79708428800  y^{32} 
+ 263302574080  y^{36} 
+ 391107684352  y^{40} 
+ \cdots,
\end{align*}
respectively.
% This was done by {\sc Magma}~\cite{Magma}.
%%%%%%%%%%%
By Corollary~\ref{cor}, 
a coset of minimum weight $14$ of an extremal doubly even
self-dual code of length $80$, if there is any, 
has the following weight enumerator:
\begin{align*}
&
3200  y^{14} 
+ 49920 y^{16} 
+ 640640  y^{18} 
+ 6429696 y^{20} 
+ 49304320  y^{22} 
\\ &
+ 295093760 y^{24} 
+ 1398270720  y^{26} 
+ 5293911040 y^{28} 
+ 16137190784  y^{30} 
\\ &
+ 39854214400 y^{32} 
+ 80135036800  y^{34} 
+ 131651287040 y^{36} 
\\ &
+ 177157460480  y^{38} 
+ 195553842176 y^{40} 
+ \cdots.
\end{align*}

%%%%%%%%%%%%%%%%%%%%%%%%%%%%%%%
\section{Non-existence of certain extremal doubly even self-dual codes
of length $24k+8$}

% In this section, 
% it is shown that there is no extremal doubly even self-dual
% $[24k+8,12k+4,4k+4]$ code with covering radius meeting the 
% Delsarte bound for $k \ge 137$.

\begin{thm}
There is no $s$-extremal singly even self-dual $[24k+8,12k+4,4k+2]$ 
code for $k\ge 137$.
Equivalently, 
there is no extremal doubly even self-dual
$[24k+8,12k+4,4k+4]$ code
with covering radius $4k+2$, meeting the Delsarte bound for $k\ge 137$.
\end{thm}
\begin{proof}
Zhang~\cite{Zhang} showed the non-existence of an extremal doubly even 
self-dual code of length $24k+8$ for $k \ge 159$.
Hence, by Theorem~\ref{thm}, 
there is no $s$-extremal singly even self-dual $[24k+8,12k+4,4k+2]$ code
for $k \ge 159$.

By using~\cite[(10), (11)]{C-S},
we numerically determined
the weight enumerators of an $s$-extremal 
singly even self-dual
$[24k+8,12k+4,4k+2]$ code $C$ and its shadow for each $k \le 158$.
% This was done by {\sc Magma}~\cite{Magma}.
This calculation was done by 
the program written in {\sc Magma}~\cite{Magma}
which is listed in~\ref{appendix}.
Then we find that
the weight enumerator of $C$ contains
a negative coefficient for $k=137,138,\ldots,158$.
% and 
% the weight enumerator of $S$ contains
% a negative coefficient for $k=156,157,158$.
The first assertion follows.
The second assertion follows from Theorem~\ref{thm}
immediately.
\end{proof}

%%%%%%%%%%%%%%%%%%%%%%%%%%%%%%%
\section{Non-existence of certain extremal doubly even self-dual codes
of length $24k+16$}

The argument in the previous section can be applied to 
$s$-extremal singly even self-dual $[24k+16,12k+8,4k+4]$
codes and extremal doubly even self-dual 
$[24k+16,12k+8,4k+4]$ codes
with covering radius $4k+4$, meeting the Delsarte bound.

\begin{lem}[Harada and Munemasa~\cite{HM07}]\label{lem:24k+16}
There is an $s$-extremal singly even self-dual $[24k+16,12k+8,4k+4]$
code if and only if there is an extremal doubly even self-dual 
$[24k+16,12k+8,4k+4]$ code
with covering radius $4k+4$, meeting the Delsarte bound.
% If $C$ is an $s$-extremal singly even self-dual $[24k+16,12k+8,4k+4]$ code.
% Then the two doubly even self-dual neighbors  $C_0 \cup C_1$ and $C_0 \cup C_3% $
% are extremal and their covering radii are $4k+4$, meeting the Delsarte bound.
% Conversely, if $D$ is an extremal doubly even self-dual
% $[24k+16,12k+8,4k+4]$ code with covering radius $4k+4$.
% Then $D$ has an $s$-extremal singly even self-dual 
% $[24k+12,12k+8,4k+4]$ neighbor.
\end{lem}

\begin{thm}
There is no
$s$-extremal singly even self-dual $[24k+16,12k+8,4k+4]$ code
for $k \ge 148$. 
Equivalently, 
there is no extremal doubly even self-dual
$[24k+16,12k+8,4k+4]$ code
with covering radius $4k+4$, meeting the Delsarte bound
for $k \ge 148$. 
\end{thm}
\begin{proof}
Zhang~\cite{Zhang} showed the non-existence of an extremal doubly even 
self-dual code of length $24k+16$ for $k \ge 164$.
Hence, by Lemma~\ref{lem:24k+16},
there is no $s$-extremal singly even self-dual $[24k+16,12k+8,4k+4]$ code
for $k \ge 164$.

By using~\cite[(10), (11)]{C-S},
we numerically determined
the weight enumerators of an $s$-extremal 
singly even self-dual
$[24k+16,12k+8,4k+4]$ code $C$ and its shadow $S$ for each $k \le 163$.
% This was done by {\sc Magma}~\cite{Magma}.
This calculation was done by slightly modifying the program
given  in~\ref{appendix} to this case.
Then we find that
% the weight enumerator of $C$ contains
% a negative coefficient for $k=154,\ldots,163$,
% and 
the weight enumerator of $S$ contains
a negative coefficient for $k=148,149,\ldots,163$.
The first assertion follows.
The second assertion follows from by Lemma~\ref{lem:24k+16}
immediately.
\end{proof}

%%%%%%%%%%%%%%%%%%%%
\bigskip
\noindent
{\bf Acknowledgment.}
This work is supported by JSPS KAKENHI Grant Number 15H03633.

%%%%%%%%%%%%%%%%%%%  References  %%%%%%%%%%%%%%%%%%%%%%%%

%%%%%%%%%%%%%%%%%%%%%%%%%%%%%%%%%

%%%%%%%%%%%%%%%%% Appendix %%%%%%%%%%%%%%%%%
%%%%%%%%%%%%%%%%% Appendix %%%%%%%%%%%%%%%%%
\appendix
\def\thesection{Appendix \Alph{section}}
\section{}\label{appendix}

\begin{verbatim}
Rxy<x,y>:=PolynomialRing(Rationals(),2);

S:=function(k)
 n:=24*k+8;
 n2:=n div 2;
 n8:=n div 8;
 d2:=2*k+1;
 basis:=[ (x^2+y^2)^(n2-4*j)*(x^2*y^2*(x^2-y^2)^2)^j 
  : j in [0..n8] ];
 basisCoeff:=Matrix(Rationals(),n8+1,d2,
  [ [ MonomialCoefficient(basis[i],x^(n-2*j)*y^(2*j))
  : j in [0..d2-1] ] : i in [1..n8+1] ] );
 sbasis:=[ (-1)^j*2^(n2-6*j)*(x*y)^(n2-4*j)*(x^4-y^4)^(2*j)
  : j in [0..n8] ];
 sbasisCoeff:=Matrix(Rationals(),n8+1,d2+1,
  [ [ MonomialCoefficient(sbasis[i],x^(n-2*j)*y^(2*j))
  : j in [0..d2] ] : i in [1..n8+1] ] );
 mat:=HorizontalJoin(basisCoeff,sbasisCoeff);
 rhsv:=Vector(Rationals(),Ncols(mat),
    [1] cat [ 0: i in [1..d2-1] ] cat [ 0: i in [0..d2] ]);
 sol,nulsp:=Solution(mat,rhsv);
 print Dimension(nulsp) eq 0;
 return [ &+[ sol[i]*basis[i] : i in [1..n8+1] ],
  &+[ sol[i]*sbasis[i] : i in [1..n8+1] ] ];
end function;

for k in [1..158] do
 WS:=S(k);
 print k,Minimum(Coefficients(WS[1]));
end for;
\end{verbatim}
%%%%%%%%%%%%%%%%% Appendix %%%%%%%%%%%%%%%%%
%%%%%%%%%%%%%%%%% Appendix %%%%%%%%%%%%%%%%%

\end{document}